\documentclass{amsart}
\usepackage{amsmath}%用于Align环境
\usepackage{amsfonts}%用于\mathbb{}
\usepackage{amsthm}%用于定理环境
\usepackage{amssymb}%第一次用于写右箭头\nrightarrow 命令
\usepackage{enumerate}%计时器，用于编号公式之类;还用于列表环境。
\usepackage{cite}
\usepackage{color}
\usepackage{xcolor}

\usepackage{hyperref}
\usepackage{fancyhdr}
\hypersetup{
	colorlinks=true,
	anchorcolor=blue,
	linkcolor=black,
	filecolor=blue,
	urlcolor=blue,	
	citecolor=blue,
	bookmarksopen=true,
	pdfborder=000
}

\theoremstyle{plain}
\newtheorem{thm}{Theorem}[section]%写定理

\newtheorem{proposition}[thm]{Proposition}

\newtheorem{cor}[thm]{Corollary}

\newtheorem{lemma}[thm]{Lemma}
\theoremstyle{definition}%定义格式
\newtheorem{defn}[thm]{Definition}
\newtheorem{exam}[thm]{Example}
\theoremstyle{remark}%注释格式
\newtheorem{rmk}[thm]{Remark}

\begin{document}
\title{Euler-symmetric complete intersections in projective space} 
\author{Zhijun Luo}
\address{School of Mathematical Sciences, University of Chinese Academy of Sciences, Beijing, 100049, China }
\email{luozj@amss.ac.cn}

\begin{abstract}
  Euler-symmetric projective varieties, introduced by Baohua Fu and Jun-Muk Hwang in 2020, are nondegenerate projective varieties admitting many $\mathbb{C}^{\times}$-actions of Euler type. They are quasi-homogeneous and uniquely determined by their fundamental forms at a general point. 
  In this paper, we study complete intersections in projective spaces which are Euler-symmetric. It is proven that such varieties are complete intersections of hyperquadrics and the base locus of the second fundamental form at a general point is again a complete intersection.
\end{abstract}
\maketitle

\section{Introduction}\label{In}
  Throughout this paper, we work over the field of complex numbers $\mathbb{C}$.
  In \cite{FH20}, the following notion is introduced as a quasi-homogeneous generalization of Hermitian symmetric spaces.
  \begin{defn}
    Let $Z \subset \mathbb{P}V$ be a projective variety. For a nonsingular point $x\in Z$, a $\mathbb{C}^{\times}$-action on $Z$ coming from a multiplicative subgroup of ${\bf G}\mathrm{L}(V)$ is said to be of \textit{Euler type} at $x$, if $x$ is an isolated fixed point of the induced action on $Z$ and the isotropic action on the tangent space $T_xZ$ is by scalar multiplication (i.e., the induced action on $\mathbb{P}T_xZ$ is trivial). We say that $Z \subset \mathbb{P}V$ is \textit{Euler-symmetric} if for a general point $x\in Z$, there exists a $\mathbb{C}^{\times}$-action on $Z$ of Euler type at $x$.
  \end{defn}
  In \cite[Proposition 2.7]{FH20}, it is shown that any Euler-symmetric projective variety is uniquely determined by their fundamental forms (See Definition \ref{fundametalforms}) at general points. 
  Conversely, given a symbol system $\textbf{F} \subset \operatorname{Sym}(W^*)$ of rank $r$ (roughly, a symbol system is a vector subspace of a polynomial ring with higher order less than rank satisfying the prolongation property. See Definition \ref{d.symbol}, \ref{rank}), 
  there exists a unique Euler-symmetric projective variety (denoted by $M(\textbf{F})$), whose fundamental forms at general points are isomorphic to $\textbf{F}$ (\cite[Theorem 3.7]{FH20}). It is a challenging problem to relate geometrical properties of $M(\textbf{F})$ to algebraic properties of $\textbf{F}$.
  
  It turns out that any Euler-symmetric projective variety $Z$ is quasi-homogeneous. More precisely, it is an equivariant compactification of a vector group (\cite[Theorem 3.7]{FH20}), namely there exists a $\mathbb{G}_a^n$-action on $Z$ with an open orbit, where $n=\dim Z$. In this way, we obtain lots of examples of equivariant compactifications of vector groups, which are in general singular.
  
  In the smooth case, there are several papers dedicated to the study of different equivariant compactification structures on a given variety.

  The first one is the work of Hassett and Tschinkel \cite{hassett1999geometry}, where they established a correspondence between equivariant compactification structures on projective space $\mathbb{P}^n$ and commutative associative local algebras with unit of dimension $n+1$. This result also follows from a more general correspondence between finite-dimensional commutative associative unital algebras and open equivariant embeddings of commutative linear algebraic groups into projective space established by Knop and Lange \cite{knop1984commutative}.

  Equivariant compactification structures on projective hypersurfaces are studied in \cite{arzhantsev2014additive}, and the case of flag varieties is studied in \cite{arzhantsev2011flag}, \cite{cheong2017equivariant} and \cite{devyatov2015unipotent}. There are also some works on toric varieties \cite{arzhantsev2017additive}, \cite{dzhunusov2020uniqueness}.

  The purpose of this article is to determine when an Euler-symmetric projective variety is a complete intersection. 
  According to Theorem 3.1 of \cite{Ben13}, a smooth complete intersection has automorphism group of positive dimension if and only if it is either a smooth hyperquadric or a smooth cubic plane curve. It follows that a smooth complete intersection is Euler-symmetric if and only if it is a smooth hyperquadric. Hence, our main focus is on singular complete intersections, which are much less studied. 
  
  Our first result is the following:
  \begin{thm}\label{thm1}
  Let $Z \subset \mathbb{P}V$ be an Euler-symmetric variety corresponding to a symbol system $\textbf{F}$ of rank $r\geq 3$. Then $Z\subset \mathbb{P}V$ is not a complete intersection.
  \end{thm}

  It follows that if $M(\textbf{F}) \subset \mathbb{P}V_{\textbf{F}}$ is an Euler-symmetric complete intersection corresponding to a symbol system $\textbf{F}$, then the rank of $\textbf{F}$ is $2$. There are examples of Euler-symmetric varieties of rank 2 which are not complete intersections (see Example \ref{exam38}). 

  Now let $\textbf{F}$ be a symbol system of rank $2$, namely $F^2$ is a vector space generated by quadratic polynomials, $Q_1, \cdots, Q_k\in \operatorname{Sym}^2W^*$, and $F^j = 0$, for all $j\geq 3$. 
  The associated Euler-symmetric variety $M(\textbf{F})$ is covered by lines. For a general point $x\in M(\textbf{F})$, let $\mathcal{L}_x(M(\textbf{F}))$ denote the variety of lines (See Definition \ref{varofline}) through $x$. 
  We denote by $\textbf{Bs}(F)\subset \mathbb{P}W$ the intersection of hyperquadrics $\{Q_1=\cdots=Q_k=0\}$.

  \begin{thm}\label{thm2}
    Let $M(\textbf{F})\subset \mathbb{P}V_{\textbf{F}}$ be an Euler-symmetric variety of dimension $n$ associated to a symbol system $\textbf{F}$ of rank $2$ with $F^2 = \left\langle Q_1, \cdots, Q_c\right\rangle$ of dimension $c$. The following statements are equivalent:
    \begin{enumerate}
      \item $M(\textbf{F})$ is a complete intersection of codimension $c$ in $\mathbb{P}V_{\textbf{F}}$;
      \item $\textbf{Bs}(F)\subset \mathbb{P}W$ is a set-theoretic complete intersection of codimension $c$;
      \item $\mathcal{L}_x(M(\textbf{F})) \subset \mathbb{P}T_xM(\textbf{F})$ is a set-theoretic complete intersection of codimension $c$, for a general point $x\in M(\textbf{F})$.
      \item The finite sequence $(Q_1, \cdots, Q_c)$ is a regular sequence in $\operatorname{Sym}(W^*)$.
    \end{enumerate}
  \end{thm}

  For the proof, first we show that an Euler-symmetric complete intersection is a complete intersection of hyperquadrics (Proposition \ref{prop31}). 
  When the rank is at least $3$, we then find one more homogeneous polynomial that, after several computations, does not lie in the ideal generated by hyperquadrics defined in Section \ref{sect3}, which proves Theorem \ref{thm1}. For Theorem \ref{thm2}, we use the relation of regular sequence (See Definition \ref{regseq}) of homogeneous polynomials with set-theoretic complete intersection to get a numerical criterion for $Y$ to be a set-theoretic complete intersection (Proposition \ref{prop41}). 
  Finally, we prove equivalent conditions for $M(\textbf{F})$ to be a complete intersection in Theorem \ref{thm42}, and $M(\textbf{F}) = Y$ is a set-theoretic complete intersection of hyperquadrics in Theorem \ref{thm41}, which allows us to conclude the proof of Theorem \ref{thm2}.

\section{Preliminaries}
  We first recall some definitions from \cite{FH20}.
  \begin{defn}
    A subvariety $Z$ of dimension $n$ in $\mathbb{P}^{n+c}$ is a \textit{complete intersection} if its vanishing ideal $I(Z)$ can be generated by $c$ elements.
  \end{defn}
  \begin{defn}
    A subvariety $Z$ of dimension $n$ in $\mathbb{P}^{n+c}$ is a \textit{set-theoretic complete intersection} if $Z$ can be written as the intersection of $c$ hypersurfaces or equivalently there are $c$ homogeneous polynomials $g_1, \cdots, g_c$ such that
    $$I(Z) = \operatorname{rad}((g_1, \cdots, g_c))\subset \mathbb{C}[w_0, w_1, \cdots, w_{n+c}].$$
  \end{defn}
  \begin{defn}\label{regseq}
    For a commutative ring $R$, a \textit{regular sequence} is a sequence $a_1, \ldots, a_d$ in $R$ such that $a_i$ is not a zero-divisor on $R/(a_1, \ldots, a_{i-1})$ for $i =1, \ldots, d$.
  \end{defn}
  \begin{defn} 
    Let $W$ be a vector space. For $w\in W$, define
    $$\iota_w: \operatorname{Sym}^{k+1}W^* \to \operatorname{Sym}^{k}W^*, \, \iota_w\varphi(w_1, \cdots, w_k) = \varphi(w, w_1, \cdots, w_k),$$
    for any $w_1, \cdots, w_k \in W$.
    By convention, we define $\iota_w(\operatorname{Sym}^0W^*) = 0$.
    For a subspace $F^k \subset \operatorname{Sym}^{k}W^*$ of symmetric $k$-linear forms on $W$, define its prolongation $\textbf{prolong}(F^k) \subset \operatorname{Sym}^{k+1}W^*$ by the following
    $$ \textbf{prolong}(F^k):= \bigcap_{w\in W}\iota_w^{-1}(F^k).$$
  \end{defn}

  \begin{rmk}
	  By Lemma 3.5 of \cite{FH20}, the restriction of $\iota_w^k$ to $F^k$ determines an element in $(F^k)^*$, which is just the map $\phi \mapsto \phi(w, \cdots, w)$. By abuse of notation, we just denote it by $\iota_w^k \in (F^k)^*$.
  \end{rmk}

  \begin{defn} \label{d.symbol}
    Let $W$ be a vector space. Fix a natural number $r$. A subspace
    $$
      \textbf{F} = \oplus_{k\geq 0}F^k \subset \operatorname{Sym}(W^*):= \oplus_{k\geq 0}\operatorname{Sym}^{k}W^*
    $$
    with
    $$
      F^0 = \mathbb{C} = \operatorname{Sym}^{0}W^*, \; F^1 = W^*,\; F^r \neq 0, \; \text{and}\; F^{r+i} = 0 \; \forall \; i \geq 1,
    $$
    is called a \textit{symbol system of rank $r$}, if $F^{k+1} \subset \textbf{prolong}(F^k)$ for each $1\leq k \leq r$.
  \end{defn}
  \begin{rmk}
      The condition $F^{k+1} \subset \textbf{prolong}(F^k)$ informally implies that all partial derivative of any elements in $F^{k+1}$ lie in $F^k$. 
      Moreover, it is essential for extending the natural  action of the vector group $W$ on $\mathbb{P}(\mathbb{C}\oplus W)$ to an action of $W$ on $\mathbb{P}V_{\textbf{F}}$. 
      For convenience, we often equate $\operatorname{Sym}(W^*)$ with $\mathbb{C}[x_1, \ldots, x_n]$, where $n = \operatorname{dim}(W)$.
  \end{rmk}
  \begin{defn}\label{fundametalforms}
    Let $x\in Z \in \mathbb{P}V$ be a nonsingular point of a nondegenerate projective variety. 
    Let $L$ be the line bundle on $Z$ given by the restriction of the hyperplane line bundle on $\mathbb{P}V$. For each nonnegative integer $k$, let $\textbf{m}_{x, Z}^k$ be the $k$-th power of the maximal ideal $\textbf{m}_{x,Z}$. For a section $s\in H^0(Z, L)$, let $j^k_x(s)$ be the $k$-jet of $s$ at $x$ such that $j^0_x(s) = s_x\in L_x$. We have a descending filtration of the dual space $V^*\subset H^0(Z, L)$ by 
    \[
      V^*\cap \operatorname{Ker}(j^k_x) \subset V^*\cap \operatorname{Ker}(j^{k-1}_x). 
    \]
    This induced homomorphism
    \[
        (V^*\cap \operatorname{Ker}(j^{k-1}_x))/(V^*\cap \operatorname{Ker}(j^k_x))\to L_x\otimes \operatorname{Sym}^kT^*_xZ
    \]
    is injective. For each $k\geq 2$, the subspace $F^k_x \subset \operatorname{Sym}^kT^*_xZ$ defined by the image of this homomorphism is called the $k$-th \textit{fundamental form} of $Z$ at $x$. For convenience, set $F^0_x = \operatorname{Sym}^0T^*_xZ = \mathbb{C}$ and $F^1_x = \operatorname{Sym}^1T^*_xZ = T^*_xZ$. The collection of the subspaces
    \[
        \textbf{F} := \oplus_{k\geq 0} F^k_x \subset \operatorname{Sym}(T^*_xZ)
    \]
    is called the \textit{system of fundamental forms} of $Z$ at $x$.
  \end{defn}

  \begin{rmk}
    On can find in \cite{ivey2003cartanivey2003cartan} a more general definition of fundamental forms. By Cartan's theorem (p.68 \cite{LM03} or Exercise 12.1.10 in \cite{ivey2003cartanivey2003cartan}),  the fundamental forms at a general point form a symbol system.
  \end{rmk}
  \begin{defn}\label{rank}
    Given a symbol system $\textbf{F}$, define a rational map
    $$
      \phi_{\textbf{F}} : \mathbb{P}(\mathbb{C}\oplus W) \dashrightarrow \mathbb{P}(\mathbb{C}\oplus W \oplus (F^2)^*\oplus \cdots \oplus (F^r)^*),
    $$
    by
    $$
      [t:w] \mapsto [t^r: t^{r-1}w: t^{r-2}\iota^2_w:\cdots: t\, \iota^{r-1}_w:\iota^r_w].
    $$
    Write $V_{\textbf{F}}:=\mathbb{C} \oplus W \oplus (F^2)^*\oplus \cdots \oplus (F^r)^*$. We will denote the closure of the image of the rational map $\phi_{\textbf{F}}$ by $M(\textbf{F}) \subset \mathbb{P}V_{\textbf{F}}$. Then $\phi_{\textbf{F}} : \mathbb{P}(\mathbb{C}\oplus W) \dashrightarrow M(\textbf{F})$ is a birational map and $M(\textbf{F}) \subset \mathbb{P}V_{\textbf{F}}$ is nondegenerate, i.e. $M(\textbf{F}) \nsubseteq H$, for any hyperplane section $H$ of $\mathbb{P}V_{\textbf{F}}$.

    We say the projective variety $M(\textbf{F})$ associated to a symbol system $\textbf{F}$ has rank $r$, denoted by $\operatorname{rank}(M(\textbf{F}))$, if the symbol system $\textbf{F}$ has rank $r$. 
    Set
    $$N = \operatorname{dim}(\mathbb{P}V_{\textbf{F}}), \; n = \operatorname{dim}(\mathbb{P}(\mathbb{C}\oplus W)) = \operatorname{dim}(M(\textbf{F})). $$
  \end{defn}

  \begin{rmk}
    Let $r = \operatorname{rank}(M(\textbf{F}))$, if $r = 1$, then $M(\textbf{F}) = \mathbb{P}^n = \mathbb{P}^N$. Hence, we always assume that $r \geq 2$.
  \end{rmk}

  \begin{thm}[Theorem 3.7 of \cite{FH20}]\label{thmfh20}
    Let $o= [1:0:\cdots:0]\in M(\textbf{F})$ be the point $\phi_{\textbf{F}}([t=1: w=0])$. Then:
    \begin{enumerate}%[(i)]
      \item The natural action of the vector group $W$ on $\mathbb{P}(\mathbb{C}\oplus W)$ can be extended to an action of $W$ on $\mathbb{P}V_{\textbf{F}}$ preserving $M(\textbf{F})$ such that the orbit of $o$ is an open subset biregular to $W$.
      \item The $\mathbb{C}^{\times}$-action on $W$ with weight $1$ induces a $\mathbb{C}^{\times}$-action on $M(\textbf{F})$ of Euler type at $o$, making $M(\textbf{F})$ Euler-symmetric.
      \item The system of fundamental forms of $M(\textbf{F})\subset \mathbb{P}V_{\textbf{F}}$ at $o$ is isomorphic to the symbol system $\textbf{F}$.
    \end{enumerate}
    Conversely, any Euler-symmetric projective variety is of the form $M(\textbf{F})$ for some symbol system $\textbf{F}$ on a vector space $W$.
  \end{thm}
  
  Recall that a smooth complete intersection has automorphism group of positive dimension if and only if it is either a smooth hyperquadric or a smooth cubic plane curve. This immediately gives the following.
  \begin{cor}
    A smooth Euler-symmetric variety is not a complete intersection unless it is a hyperquadric.
  \end{cor}

  There are lots of Euler-symmetric varieties by Theorem \ref{thmfh20}. The example below from Example 2.2 \cite{FH20} shows that there are at least as many nonsingular Euler-symmetric varieties as nonsingular projective varieties.
  \begin{exam}
    Let $S\subset \mathbb{P}^{n-1} \subset \mathbb{P}^n$ be a nonsingular algebraic subset in a hyperplane of $\mathbb{P}^n$. For each point $x\in \mathbb{P}^n\backslash \mathbb{P}^{n-1}$, the scalar multiplication on the affine space $\mathbb{P}^n\backslash \mathbb{P}^{n-1}$ regarded as a vector space with the origin at $x$ can be extended to a $\mathbb{C}^{\times}$-action
    \[
      \textbf{A}_x: \mathbb{C}^{\times}\times \mathbb{P}^n \to \mathbb{P}^n,
    \]
    which fixes every point of the hyperplane $\mathbb{P}^{n-1}$. Let $\beta:\textbf{Bl}_S(\mathbb{P}^n)\to \mathbb{P}^n$ be the blowup of $\mathbb{P}^n$ along $S$ and let $E$ the exceptional divisor. For suitable positive integers $a$ and $b$, the line bundle $L := \mathcal{O}(-aE)\otimes\beta^*\mathcal{O}_{\mathbb{P}^n}(b)$ on $\textbf{Bl}_S(\mathbb{P}^n)$ is very ample. The action $\textbf{A}_x$ induces an action on the image
    \[
      Z \subset \mathbb{P}H^0(\textbf{Bl}_S(\mathbb{P}^n), L)^*
    \]
    of the projective embedding induced by line bundle $L$, which is of Euler type at $x\in Z$. Thus, $Z$ is an Euler-symmetric projective variety.
  \end{exam}

  \begin{exam}
    Example 3.14 of \cite{FH20} shows that a nondegenerate hypersurface is Euler-symmetric if and only if it is a nondegenerate hyperquadric.
  \end{exam}

  \begin{defn}
    For a symbol system $\textbf{F} = \oplus_{k\geq 0}F^k$ of rank $r$, define the projective algebraic subset $\textbf{Bs}(F^k) \subset \mathbb{P}W$ by the affine cone in $W$
    \[
      \{w\in W\mid \phi(w, \cdots, w) = 0,\; \forall \phi \in F^k\}
    \]
    By the definition of a symbol system, we have the inclusion $ \textbf{Bs}(F^k) \subset \textbf{Bs}(F^{k+1})$ for each $k\in \mathbb{N}$. The \textit{base loci} of $\textbf{F}$ is the nonempty projective algebraic subset
    \[
      \textbf{Bs}(\textbf{F}) = \bigcap_{\textbf{Bs}(F^k) \neq \emptyset, k \leq r+1}\textbf{Bs}(F^k).
    \]
  \end{defn}

  \begin{defn}\label{varofline}
    Let $Z\subset \mathbb{P}^N$ be a connected equidimensional nondegenerate projective variety of dimension $n\geq 1$, $x \in Z$ a general smooth point, the variety $$\mathcal{L}_x(Z) = \{[l]\mid x\in l \subset Z, l \text{ is line in } \mathbb{P}^N\} $$ is called the \textit{variety of lines} through $x$. 
    Note that $\mathcal{L}_x(Z)$ is naturally embedded in $\mathbb{P}^{n-1}$ which is the space of tangent directions at $x$. For more details about variety of lines, we refer readers to Section 2.2 of \cite{russo2016geometry}.
  \end{defn}

\section{Euler-symmetric varieties of rank $\geq 3$}\label{sect3}
  Let $W$ be a vector space of dimension $n$, and let $\textbf{F} = \oplus_{k=0}^r F^k\subset \operatorname{Sym}(W^*)$ be a symbol system of rank $r\geq 3$, where $F^k = \left\langle Q^{(k)}_1, \cdots, Q^{(k)}_{s_k} \right\rangle$ is the vector subspace of $\operatorname{Sym}^kW^*$ generated by $Q^{(k)}_1, \cdots, Q^{(k)}_{s_k}$, for $r\geq k\geq 2$ and $F^1 = W^*$.
  Write the coordinates of $\mathbb{P}V_{\textbf{F}}$ as $[z_0:z_1:\cdots:z_n:w^{(2)}_1:\cdots:w^{(2)}_{s_2}:\cdots:w^{(r)}_1:\cdots:w^{(r)}_{s_r}]$, and let
  $$\mathcal{I} =I(M(\textbf{F})) \subset \mathbb{C}[z_0, z_1, \cdots, z_n, w^{(2)}_1, \cdots, w^{(2)}_{s_2}, \cdots, w^{(r)}_1, \cdots, w^{(r)}_{s_r}]$$
  be the defining ideal of $M(\textbf{F})$.
  By definition \ref{rank}, the codimension of $M(\textbf{F})$ in $\mathbb{P}V_{\textbf{F}}$ is $m = N - n = \sum_{i=2}^r \operatorname{dim}(F^i) = \sum_{i =2}^r s_i$.

  By the prolongation property of Definition \ref{d.symbol}, the following $m$ polynomials of degree $2$ lie in $\mathcal{I}$:
  \begin{align*}
    f_1^{(2)} &= z_0w_1^{(2)} - Q_1^{(2)}(z_1, \cdots, z_n);\\
      & \vdots \\
    f_{s_2}^{(2)} &= z_0w_{s_2}^{(2)} - Q_{s_2}^{(2)}(z_1, \cdots, z_n);\\
      & \vdots \\
    f_{i}^{(j)} &= z_0w_{i}^{(j)} - \sum_{l = 1}^{s_{j-1}}g_l^{(j-1)}(z_1, \cdots, z_n) w_l^{(j-1)} ;\\
  \end{align*}
  where $2 < j \leq r$, $1\leq i \leq s_j$ and $g_l^{(j-1)}$ are linear polynomials satisfying the following Euler identity:
  \[
      Q^{(j)}_i = \sum_{l = 1}^{s_{j-1}}g_l^{(j-1)} Q_l^{(j-1)}.
  \]

  \begin{proposition}\label{prop31}
    An Euler-symmetric complete intersection is a complete intersection of hyperquadrics defined as above.
  \end{proposition}
  \begin{proof}
    Let $M(\textbf{F})\subset \mathbb{P}V_{\textbf{F}}$ be an Euler-symmetric complete intersection corresponding to the symbol system $\textbf{F}$. Since $M(\textbf{F})\subset \mathbb{P}V_{\textbf{F}}$ is nondegenerate, all polynomials $g \in\mathcal{I}$ have $\operatorname{deg}(g) \geq 2$. Since $M(\textbf{F})$ is a complete intersection, there exist $m = \operatorname{codim}_{\mathbb{P}V_{\textbf{F}}}(M(\textbf{F}))$ homogeneous polynomials $g_i\in \mathcal{I}$ such that
    \[
      \mathcal{I} = (g_1, \cdots, g_m).
    \]
    Therefore, we have the inclusion
    \[
      \left\langle f_i^{(j)}\mid 2 \leq j \leq r,\; 1\leq i \leq s_j \right\rangle \subset \left\langle g_1\cdots, g_m \right\rangle,
    \]
    where $\left\langle S \right\rangle$ is the vector space generated by elements in $S$. By dimension reason, the inclusion is an equality, i.e., $\mathcal{I} =(f_i^{(j)}\mid 2 \leq j \leq r,\; 1\leq i \leq s_j))$.
  \end{proof}

  Let $\mathcal{J}_1$ be the ideal generated by $f_i^{(j)},\; 2 \leq j \leq r,\; 1\leq i \leq s_j$, $\mathcal{J} = \operatorname{rad}(\mathcal{J}_1)\subset \mathcal{I}$.

  \begin{lemma}\label{lemma31}
    Let $G$ be a nonzero homogeneous polynomial in
    $$\mathbb{C}[z_1, \cdots, z_n, w^{(r-1)}_1, \cdots, w^{(r-1)}_{s_{r-1}}, w_1^{(r)}]$$ such that $G = G_0 + w_1^{(r)}G_1$ and $0 \neq G_0 \in \mathbb{C}[w^{(r-1)}_1, \cdots, w^{(r-1)}_{s_{r-1}}]$.
    If $G \in \mathcal{I}$ and $s_r = 1$, then $G \notin \mathcal{J}$.
  \end{lemma}
  \begin{proof}
    Define $\operatorname{deg}_t(w_i^{(j)}) = r-j$, $\operatorname{deg}_t(z_i) = r-1$, $\operatorname{deg}_t(z_0) = r$. If $G \in \mathcal{J}$, then there exists a positive integer $k$ such that $F = G^k = F_0 + w_1^{(r)}F_1 \in \mathcal{J}_1$, where $0 \neq F_0 \in \mathbb{C}[w^{(r-1)}_1, \cdots, w^{(r-1)}_{s_{r-1}}]$. Let $\operatorname{deg}(F_0) = s$, then $\operatorname{deg}_t(F_0) = s$ and $\operatorname{deg}_t(f_i^{(j)}) = 2r-j$.
    Therefore, there exist $m$ homogeneous polynomials $h_i^{(j)}$ of degree $s-2$ in $\mathbb{C}[z_0, z_1, \cdots, z_n, w^{(2)}_1, \cdots, w^{(2)}_{s_2}, \cdots, w^{(r-1)}_1, \cdots, w^{(r-1)}_{s_{r-1}}, w_1^{(r)}]$ such that
    \[
       F = \sum_{i, j} h_i^{(j)} f_i^{(j)}.
    \]
    Let $f_i^{(j)} = z_0w_i^{(j)} + \overline{f_i^{(j)}} $, $h_i^{(j)} = \overline{h_i^{(j)}} + z_0 e_i^{(j)} $, where $$\overline{h_i^{(j)}} \in \mathbb{C}[z_1, \cdots, z_n, w^{(2)}_1, \cdots, w^{(2)}_{s_2}, \cdots, w^{(r)}_1].$$
    Then we have
    \begin{align*}
      F = & \sum_{i, j} h_i^{(j)} f_i^{(j)} \\
        = & \sum_{i,j} \overline{h_i^{(j)}} \cdot \overline{f_i^{(j)}} + z_0\sum_{i, j}(z_0e_i^{(j)}w_i^{(j)} + \overline{h_i^{(j)}} w_i^{(j)} + \overline{f_i^{(j)}} e_i^{(j)}).
    \end{align*}
    Since
    \[
      F,\; \overline{h_i^{(j)}},\; \overline{f_i^{(j)}} \in \mathbb{C}[z_1, \cdots, z_n, w^{(2)}_1, \cdots, w^{(2)}_{s_2}, \cdots, w^{(r)}_1],
    \]
    then $F = \sum_{i, j} \overline{h_i^{(j)}} \cdot \overline{f_i^{(j)}}$.

    Let $\overline{h_i^{(j)}}= H_i^{(j)} + w_1^{(r)}l_i^{(j)}$, where $$H_i^{(j)} \in \mathbb{C}[z_1, \cdots, z_n, w^{(2)}_1, \cdots, w^{(2)}_{s_2}, \cdots, w^{(r-1)}_{s_{r-1}}].$$
    Hence,
    \begin{align*}
      F = & F_0 + w_1^{(r)}F_1 \\
        = & \sum_{i, j} \overline{h_i^{(j)}}\cdot \overline{f_i^{(j)}} \\
        = & \sum_{i,j} H_i^{(j)} \overline{f_i^{(j)}} + w_1^{(r)}\sum_{i, j}(l_i^{(j)}\overline{f_i^{(j)}} ).
    \end{align*}
    This implies $F_0 = \sum_{i, j} H_i^{(j)} \overline{f_i^{(j)}}$.

    Let $H_i^{(j)} = \sum_{a} H_{i, a}^{(j)}$ be the homogeneous decomposition with respect to $t$, where $H_{i, a}^{(j)}$ is a homogeneous polynomial, and $\operatorname{deg}_t(H_{i, a}^{(j)}) = a$.
    Note that $H_{i, a}^{(j)}$ are also homogeneous polynomials of degree $s-2$ and
    \[
      H_{i, a}^{(j)} \in \mathbb{C}[z_1, \cdots, z_n, w^{(2)}_1, \cdots, w^{(2)}_{s_2}, \cdots, w^{(r-1)}_{s_{r-1}}].
    \]
    This implies $a = \operatorname{deg}_t(H_{i, a}^{(j)}) \geq s-2$.
    Then
    \begin{align*}
      F_0 =& \sum_{i, j} H_i^{(j)} \overline{f_i^{(j)}}\\
         = &  \sum_{i, j} \sum_a H_{i, a}^{(j)} \overline{f_i^{(j)}},
    \end{align*}
    where $\operatorname{deg}_t(H_{i, a}^{(j)} \overline{f_i^{(j)}}) \geq s-2 + 2r-j$ for every $i, j$. Since $r\geq 3$, we have $s-2 + 2r-j \geq s + 1 > s$. Then
    $$s = \operatorname{deg}_t(F_0) =\operatorname{deg}_t(\sum_{i, j, a} H_{i, a}^{(j)} \overline{f_i^{(j)}}) > s ,$$ which is a contradiction.
  \end{proof}

  \begin{lemma}\label{lemma32}
    Let $F\in \mathbb{C}[x_1, \cdots, x_n]$ be a homogeneous polynomial of degree $m$. If $V = \left\langle \partial_{x_1}F, \cdots, \partial_{x_n}F\right\rangle$ is a vector space of dimension $k\leq n$, then there exist $k$ linearly independent linear homogeneous polynomials $f_1, \cdots, f_k$ such that :
    \[
      F \in \mathbb{C}[f_1, \cdots, f_k]_m
    \]
    where $\mathbb{C}[f_1, \cdots, f_k]_m$ is the collection of all homogeneous polynomials of degree $m$ with respect to $f_1, \cdots, f_k$.
  \end{lemma}
  \begin{proof}
    Write $\partial_{j}F = \partial_{x_j}F$. We may assume that $V = \left\langle \partial_1F, \cdots, \partial_kF\right\rangle$. By Euler identity, we have
    \[
      mF = \sum_{i=1}^{n}x_i \partial_iF = \sum_{i=1}^kx_i\partial_iF + \sum_{j=k+1}^nx_j\partial_jF.
    \]
    Since $\partial_jF \in V$, for $k+1 \leq j \leq n$, then $\partial_jF = \sum_{i=1}^ka_{ij}\partial_iF$. Therefore, we have
    \[
        mF = \sum_{i=1}^{k}f_i \partial_iF, \quad f_i = x_i + \sum_{j>k}a_{ij}x_j,
    \]
    and $f_i$ are linearly independent linear polynomials.

    We claim $$(m-s)\partial_{i_1}\partial_{i_2}\cdots\partial_{i_s} F = \sum_{j=1}^{k}f_j \partial_{i_1}\partial_{i_2}\cdots\partial_{i_{s}}\partial_j F,$$ for all $s < m$, $1 \leq i_1, \ldots, i_s \leq k$. 

    For $s = 1$, we have
    \begin{align*}
      m\partial_iF &= \sum_{j =1}^{k}(f_j\partial_i(\partial_jF) + \partial_i(f_j)\partial_jF);\\
      & = \sum_{j =1}^{k}(f_j\partial_i(\partial_jF) + \delta_{ij}\partial_jF); \\
      & =\partial_iF + \sum_{j =1}^{k}f_j\partial_i(\partial_jF).
    \end{align*}
    By induction, we have,
    \begin{align*}
      & (m-s)\partial_{i_{s+1}}\partial_{i_1}\partial_{i_2}\cdots\partial_{i_{s}} F \\
      = &\sum_{j=1}^{k}\partial_{i_{s+1}}(f_j \partial_{i_1}\partial_{i_2}\cdots\partial_{i_{s}}\partial_j F); \\
      = & \sum_{j=1}^k(f_j\partial_{i_{s+1}}\partial_{i_1}\partial_{i_2}\cdots\partial_{i_{s}}\partial_j F + (\partial_{i_{s+1}}f_j) \partial_{i_1}\partial_{i_2}\cdots\partial_{i_{s}}\partial_j F) \\
      = & \sum_{j=1}^k(f_j\partial_{i_{s+1}}\partial_{i_1}\partial_{i_2}\cdots\partial_{i_{s}}\partial_j F + \delta_{ji_{s+1}} \partial_{i_1}\partial_{i_2}\cdots\partial_{i_{s}}\partial_j F) \\
      = & \partial_{i_1}\cdots\partial_{i_{s}}\partial_{i_{s+1}} F + \sum_{j=1}^kf_j\partial_{i_{s+1}}\partial_{i_1}\partial_{i_2}\cdots\partial_{i_{s}}\partial_j F, \\
    \end{align*}
    where $1\leq i_{s+1}\leq k$, $s+1 < m$. This proves the claim.

    As $\operatorname{deg}F = m$, $\partial_{i_1}\partial_{i_2}\cdots\partial_{i_{m-1}}\partial_jF$ are constant for all $1 \leq i_1, \ldots, i_{m-1}, j \leq k$.
    Then by the claim,
    \[
      \partial_{i_1}\partial_{i_2}\cdots\partial_{i_{m-1}} F = \sum_{j=1}^{k}f_j \partial_{i_1}\partial_{i_2}\cdots\partial_{i_{s}}\partial_jF \in \mathbb{C}[f_1, \cdots, f_k]_1.
    \]
    Therefore, $F \in \mathbb{C}[f_1, \cdots, f_k]_m$ by induction which concludes the proof.
  \end{proof}

  Now, we can prove the first main theorem of this paper.
  \begin{proof}[Proof of Theorem \ref{thm1}]
    Assume $\operatorname{rank}(\textbf{F}) = r \geq 3$. Consider the Euler-symmetric projective variety $Z = M(\textbf{F})$ and set $m = \operatorname{codim}_{\mathbb{P}V_{\textbf{F}}}(M(\textbf{F}))$. Since Euler-symmetric projective varieties are nondegenerate, all homogeneous polynomials $g \in \mathcal{I}$ have $\operatorname{deg}(g) \geq 2$.
    Let $S$ be a finite generating set of $\mathcal{I}$, i.e., $\mathcal{I}= (S)$.  And let $\left\langle S \right\rangle$ be the vector space generated by elements in $S$, then $f_i^{(j)} \in \left\langle S \right\rangle$. Thus, $\operatorname{dim}(\left\langle S \right\rangle) \geq m$. If $\operatorname{dim}(\left\langle S \right\rangle) > m$ for all finite generating sets $S$ of $\mathcal{I}$, then $M(\textbf{F})$ is not a complete intersection.
	  Thus, if we can find one nonzero homogeneous polynomial $G \in \mathcal{I} \backslash \mathcal{J}$, then $\operatorname{dim}(\left\langle S \right\rangle) > m$ for all generating sets of $\mathcal{I}$ with finite elements. 
    
    Moreover, if we can find one nonzero homogeneous polynomial 
    \[
      G \in \mathcal{I} \cap \mathbb{C}[z_0, z_1, \cdots, z_n, w^{(2)}_1, \cdots, w^{(2)}_{s_2}, \cdots, w^{(r-1)}_{s_{r-1}}, w^{(r)}_{1}],
    \]
    and 
    \[
      G \notin \mathcal{J} \cap\mathbb{C}[z_0, z_1, \cdots, z_n, w^{(2)}_1, \cdots, w^{(2)}_{s_2}, \cdots, w^{(r-1)}_{s_{r-1}}, w^{(r)}_{1}],
    \]
    then $G \notin \mathcal{J}$.
    Therefore, we can also assume that $F^r = \left\langle T \right\rangle$, where $T\in \operatorname{Sym}^rW^*$.

    If $\operatorname{dim}(M(\textbf{F})) = 1$, then $F^r = \left\langle x^r \right\rangle$ and $M(\textbf{F})$ is a rational normal curve in $\mathbb{P}V_{\textbf{F}}$, namely the closure of $\{[t^r:t^{r-1}x:\cdots:tx^{r-1}:x^r]\in \mathbb{P}V_{\textbf{F}} \}$, which is not a complete intersection as $r\geq 3$.

    If $\operatorname{dim}(M(\textbf{F})) = n \geq 2$, then we have the following cases:
    \begin{enumerate}
      \item[Case (1)] $\operatorname{dim}(F^{r-1}) = s_{r-1} > n$: \\
      we know that $\{Q^{(r-1)}_1, \cdots, Q^{(r-1)}_{s_{r-1}}\}$ are algebraically dependent, hence there exists a polynomial $G \in \mathbb{C}[y_1, \cdots, y_{s_{r-1}}]$ such that
      $$G(Q^{(r-1)}_1, \cdots, Q^{(r-1)}_{s_{r-1}}) \equiv 0.$$
      Since $Q^{(r-1)}_1, \cdots, Q^{(r-1)}_{s_{r-1}}$ are all homogeneous polynomials of degree $r-1$, we can assume that $G \in \mathbb{C}[y_1, \cdots, y_{s_{r-1}}]$ is also homogeneous. Hence, from Lemma \ref{lemma31} we have a polynomial $G(w^{(r-1)}_1, \cdots, w^{(r-1)}_{s_{r-1}}) \in \mathcal{I}\backslash \mathcal{J}$.

      \item[Case (2)] $\operatorname{dim}(F^{r-1}) = s_{r-1} = n$: \\
      Since $\{Q^{(r-1)}_1, \cdots, Q^{(r-1)}_n, T\}$ is algebraically dependent, then there exists a polynomial $G\in \mathbb{C}[y_1, \cdots, y_n, y]$ such that
      $$G(Q^{(r-1)}_1, \cdots, Q^{(r-1)}_n, T) \equiv 0.$$
      Rewrite $G$ as
      \[
        G = a_my^m + \cdots + a_1y + a_0,
      \]
      where $a_j$ are polynomials of $y_1, \cdots, y_n$. Define $\operatorname{deg}_xy_i = r-1,\; \operatorname{deg}_xy = r$. Then $\operatorname{deg}_x G$ makes sense, and we say it is the degree of $G$ w.r.t $x_1, \cdots, x_n$.

      Firstly, we will reduce $G$ to a homogeneous polynomial w.r.t $x_1, \cdots, x_n$ and $a_j$ to homogeneous polynomials.
      Let $a_j = \sum a_j^i$ be the homogeneous decomposition in $\mathbb{C}[y_1, \cdots, y_n]$.
      Since $y_1, \cdots, y_n$ are homogeneous polynomials of degree $r-1$ w.r.t. $x_1, \cdots, x_n$, $a_j^i$ are homogeneous w.r.t $x_1, \cdots, x_n$.
      Then we have the following homogeneous decomposition w.r.t $x_1, \cdots, x_n$:
      \begin{align*}
        &G(y_1, \cdots, y_n, y) \\
        =& \sum a_m^i y^m + \cdots + \sum a_1^i y + \sum a_0^i \\
        =& G_1 + G_2 + \cdots
      \end{align*}
      where $G_i = a_m^{i_m}y^m + \cdots + a_1^{i_1}y + a_0^{i_0}$ are homogeneous polynomials w.r.t $x_1, \cdots x_n$ and $a_j^{i_j}$ are homogeneous polynomials

      As $G(Q^{(r-1)}_1, \cdots, Q^{(r-1)}_n, T)$ is a zero polynomial, we get $$G_i(Q^{(r-1)}_1, \cdots, Q^{(r-1)}_n, T)$$ are also zero polynomials. Therefore, we can assume that $G$ is homogeneous w.r.t $x_1, \cdots, x_n$ and $a_j$ are homogeneous polynomials.

      Secondly, without loss of generality, we may assume $a_0 \neq 0$, otherwise, we consider $\tilde{G} = a_my^{m-1} + \cdots + a_1$, then $\tilde{G}(Q^{(r-1)}_1, \cdots, Q^{(r-1)}_n, T)$ is also a zero polynomial.

      Finally, let $h = \operatorname{deg}_x(G)$, then $(r-1)| h$ as $a_0\neq 0$, and $h = (r-1)s$, where $s = \operatorname{deg}(a_0)$. This gives
      $$h =(r-1)s = \operatorname{deg}_x(a_jy^j) = jr+(r-1)\operatorname{deg}(a_j).$$
      Hence, $(r-1)|j$, and
      \begin{align*}
        G = a_{m(r-1)} y^{m(r-1)} + \cdots + a_{2(r-1)}y^{2(r-1)} + a_{r-1}y^{r-1} + a_0.
      \end{align*}
      Therefore,
      \begin{align*}
        h = & (r-1)s = \operatorname{deg}_x(a_{j(r-1)}y^{j(r-1)}) \\
          = &j(r-1)r+(r-1)\operatorname{deg}(a_{j(r-1)}).
      \end{align*}
      Then $\operatorname{deg}(a_{j(r-1)}) = s - jr\geq 0$, for $0\leq j \leq m$. 
      By multiplying $y_1^i$ with $G$, 
      we can assume that $s = rp$, then $h = (r-1)s = (r-1)rp$ and $s - jr = rp - jr = r(p-j) \geq 0$. Note that
      \begin{align*}
        & a_{j(r-1)}(Q^{(r-1)}_1, \cdots, Q^{(r-1)}_n)T^{j(r-1)} \\
        =& a_{j(r-1)}(\frac{w^{(r-1)}_1}{t}, \cdots, \frac{w^{(r-1)}_n}{t})(w^{(r)}_1)^{j(r-1)} \\
        =& \frac{1}{t^{r(p-j)}}a_{j(r-1)}(w^{(r-1)}_1, \cdots, w^{(r-1)}_n)(w^{(r)}_1)^{j(r-1)},
      \end{align*}
      which yields that
      \begin{align*}
        & t^{rp}G(w^{(r-1)}_1, \cdots, w^{(r-1)}_n, w^{(r)}_1) \\
        =& t^{rm}a_{m(r-1)}(w^{(r-1)}_1, \cdots, w^{(r-1)}_n)(w^{(r)}_1)^{m(r-1)}+ \cdots \\
        &+ a_{0}(w^{(r-1)}_1, \cdots, w^{(r-1)}_n).
      \end{align*}
      Set
      \begin{align*}
        G^{\prime}
        = & z_0^ma_{m(r-1)}(w^{(r-1)}_1, \cdots, w^{(r-1)}_n)(w^{(r)}_1)^{m(r-1)}+\cdots \\
        &+ a_{0}(w^{(r-1)}_1, \cdots, w^{(r-1)}_n),
      \end{align*}
      $G^{\prime}$ is a homogeneous polynomial in $\mathbb{C}[z_0:z_1:\cdots:z_n:w^{(r-1)}_1:\cdots:w^{(r)}_{1}]$, and $G^{\prime} \in \mathcal{I}$. Then
      \begin{align*}
        G^{\prime}
        = & a_{m(r-1)}(w^{(r-1)}_1, \cdots, w^{(r-1)}_n)(f_1^{(r)} - \overline{f_1^{(r)}})^m(w^{(r)}_1)^{m(r-2)}+ \\
          &\vdots \\
          &+ a_{0}(w^{(r-1)}_1, \cdots, w^{(r-1)}_n)\\
        =& G_0 + f_1^{(r)}G_1 \equiv G_0 (\operatorname{mod} \mathcal{J}) ,
      \end{align*}
      where $G_0= a_{0} + w_1^{(r)}\overline{G_{0}} \in \mathcal{I}\cap \mathbb{C}[z_1, \cdots, z_n, w^{(r-1)}_1, \cdots, w^{(r-1)}_{n}, w^{(r)}_{1}]$ is a nonzero homogeneous polynomial and $\overline{f_1^{(r)}} = f_1^{(r)} - z_0w_1^{(r)}$.

      Since $a_{0}(w^{(r-1)}_1, \cdots, w^{(r-1)}_n) \neq 0$, we have $G_0 \notin \mathcal{J}$ by Lemma \ref{lemma31}.

      \item[Case (3)] $\operatorname{dim}(F^{r-1}) = s_{r-1} < n$: \\
      We have $\left\langle \partial_{x_1}T, \cdots, \partial_{x_n}T\right\rangle \subset F^{r-1}$,
      \[
        \operatorname{dim}(\left\langle \partial_{x_1}T, \cdots, \partial_{x_n}T\right\rangle) = q \leq s_{r-1} < n.
      \]
      Assume that $\left\langle \partial_{x_1}T, \cdots, \partial_{x_n}T\right\rangle = \left\langle Q^{(r-1)}_1, \cdots, Q^{(r-1)}_q \right\rangle$. By Lemma \ref{lemma32} we can assume that
      \[
        Q^{(r-1)}_1, \cdots, Q^{(r-1)}_q, T \in \mathbb{C}[x_1, \cdots, x_q].
      \]
      Similarly, by above arguments, there exists a homogeneous polynomial $G$ such that
      $$G(z_1, \cdots, z_n, w^{(r-1)}_1, \cdots, w^{(r-1)}_q, w^{(r)}_1) \in \mathcal{I} \backslash \mathcal{J}.$$
    \end{enumerate}
  \end{proof}

  \begin{exam}
    Let $F^2 =\left\langle x_1^2 + x_2^2, x_1x_2 \right\rangle $, $F^3 = \left\langle x_1^3 + 3 x_1x_2^2 \right\rangle$, and $\textbf{F} = F^0 \oplus F^1 \oplus F^2 \oplus F^3$. 
    By a direct calculation, polynomials of degree $2$ in $\mathcal{I}$ are exactly in $\left\langle f_1^{(2)}, f_2^{(2)}, f_1^{(3)} \right\rangle$,
    and $\mathcal{J} = (f_1^{(2)}, f_2^{(2)}, f_1^{(3)}, g)$, where $g = 3z_2w_1^{(2)}w_2^{(2)} - 2z_1(w_2^{(2)})^2 - z_2^2w_1^{(3)}$, then $M(\textbf{F})$ is not a complete intersection without finding one more homogeneous polynomial in $\mathcal{I} \backslash \mathcal{J}$.
  \end{exam}

  \begin{exam}
    Consider the homogeneous polynomial $P = \sum_{i=1}^n x_i^3 \in Sym^3W^*$, the symbol system $\textbf{F}_P = F^0\oplus F^1 \oplus F^2 \oplus F^3 $, where $F^2 = \left\langle x_1^2, \ldots, x^2_n\right\rangle$ and $F^3 = \left\langle P\right\rangle$. 
    The Euler-symmetric projective variety $M(\textbf{F}_P)$ defined in Example 3.10 of \cite{FH20} is exactly the projective Legendrian variety studied in Section 4.3 of \cite{LM07}. 
    The polynomial $G$ in case $(2)$ of the proof of Theorem \ref{thm1} is hard to construct. 
    However, there is a nonzero homogeneous polynomial of degree $n+1$ in $\mathcal{J}\backslash \mathcal{J}_1$, which also shows that $M(\textbf{F}_P)$ is not a complete intersection.
  \end{exam}

  \begin{exam}
    The twisted cubic curve in $\mathbb{P}^3$ is an Euler-symmetric projective variety of rank $3$, hence not a complete intersection, and its vanishing ideal can be generated by three polynomials of degree $2$.
    The rational normal curve $C^r \subset \mathbb{P}^r$ with $r\geq 3$ is an Euler-symmetric projective variety of rank $r$, hence not a complete intersection. But in 1941 Perron observed that all rational normal curves are set-theoretic complete intersections. For more details, we refer readers to \cite{buadescu2010grothendieck}, \cite{torrente2015rational}.
  \end{exam}

  \begin{cor}\label{bzy}
    Let $Z \subset \mathbb{P}^N$ be an Euler-symmetric variety of dimension $n$. If
    \[
      n < \frac{N}{2},
    \]
    then $Z$ is not a complete intersection.
  \end{cor}
  \begin{proof}
    If $\operatorname{rank}(Z) \geq 3$, then by Theorem \ref{thm1}, $Z$ is not a complete intersection. Therefore, we may assume that $\operatorname{rank}(Z) = 2$, then by Theorem \ref{thmfh20} the corresponding symbol system $\textbf{F} = F^0\oplus F^1 \oplus F^2$ and $s_2 = \operatorname{dim}(F^2) = \operatorname{codim}_{\mathbb{P}^N}(Z) = N - n > n$. From Case $(1)$ of the proof of Theorem \ref{thm1}, we have a homogeneous polynomial
    $$G \in \mathbb{C}[w_1^{(2)}, \cdots, w_{s_2}^{(2)}]$$
    such that $G \in \mathcal{I}$ but $G \notin \mathcal{J}$, where $\mathcal{J} = \operatorname{rad}((f_i^{(2)}\mid 1\leq i \leq s_2))$. Therefore, $Z$ is not a complete intersection either.
  \end{proof}

  \begin{rmk}
    The inequality in Corollary \ref{bzy} is optimal, namely, for every $n \geq \frac{N}{2}$, there exists an Euler-symmetric projective variety of dimension $n$ which is a complete intersection (Example \ref{exam41}).
  \end{rmk}
  
  \begin{exam}\label{exam38}
    Let $\textbf{F}$ be a symbol system of rank $2$, and let $F^2\subset \operatorname{Sym}^2W^*$ be a subspace of dimension $k$. If $k > n= \operatorname{dim}(W)$, then the corresponding Euler-symmetric variety $M(\textbf{F})$ is not a complete intersection by Corollary \ref{bzy}.
  \end{exam}

\section{Euler-symmetric varieties of rank $2$}

  Now we restrict ourselves to the case where $\textbf{F}$ is a symbol system of rank $2$. Let $F^2 = \left\langle Q_1, \cdots, Q_c \right\rangle$ and $\textbf{Bs}(\textbf{F})= V_+(Q_1, \cdots, Q_c)\subset \mathbb{P}^{n-1} = \mathbb{P}W$.
  Write the coordinates of $\mathbb{P}V_{\textbf{F}}$ as $[w_0:w_1:\cdots:w_n:w_{n+1}:\cdots:w_{N}]$ with $N=n+c$.

  Let
  \begin{align*}
    f_1 &= w_0w_{n+1} - Q_1(w_1, \cdots, w_n);\\
     & \vdots \\
    f_c &= w_0w_{N} - Q_c(w_1, \cdots, w_n);\\
  \end{align*}
  be $c$ polynomials of degree $2$, and let $Y$ be the subvariety of $\mathbb{P}V_{\textbf{F}}$ defined by those polynomials, 
  then $\operatorname{dim}(Y)\geq n$. Let $\mathcal{I} = I(M(\textbf{F}))$, $\mathcal{J} = I(Y)$ and $b = \operatorname{codim}_{\mathbb{P}W}(\textbf{Bs}(\textbf{F}))$. Corollary \ref{bzy} suggests that we can assume that $c \leq n$, and we always have $b \leq c$.

  For $c = n$, we say $\textbf{Bs}(\textbf{F})$ is a set-theoretic complete intersection of codimension $c$ if $\textbf{Bs}(\textbf{F}) = \emptyset$. Denote $\operatorname{codim}_{\mathbb{P}^{n-1}}(\textbf{Bs}(\textbf{F})) = n$ and $\operatorname{dim}(\textbf{Bs}(\textbf{F})) = -1$ if $\textbf{Bs}(\textbf{F}) = \emptyset$.
  \begin{defn}
    We say an Euler-symmetric projective variety $M(\textbf{F})$ is {\em quadratic}, if $M(\textbf{F}) = Y = V_+(f_1, \cdots, f_c)$ or equivalently $\mathcal{I} = \mathcal{J}$.
  \end{defn}

  \begin{lemma}\label{lemma41}
    If the subvariety $X = V_+(g_1, \cdots, g_c)\subset \mathbb{P}^{n+c}$ is a set-theoretic complete intersection of codimension $c$, then the sequence $(g_1, \cdots, g_c)$ of homogeneous polynomials $$g_i \in \mathbb{C}[w_0, w_1, \cdots, w_n, w_{n+1}, \cdots, w_{N}]$$ is a regular sequence.
  \end{lemma}
  \begin{proof}
    If the sequence $(g_1, \cdots, g_c)$ is not regular, then there exists $1 \leq j< c$ such that the sequence $(g_1, \cdots, g_j)$ is regular, while the sequence $(g_1, \cdots, g_j, g_{j+1})$ is not. 
    Therefore, every irreducible component of $V_+(g_1, \cdots, g_j)$ has dimension $n + c -j$, and there exists an irreducible component $Z_1$ of $V_+(g_1, \cdots, g_j, g_{j+1})$ with dimension $n +c-j$. Then every irreducible component of $Z_1\cap V_+(g_{j+2})\subset V_+(g_1, \cdots, g_j, g_{j+1}, g_{j+2})$ has dimension $\geq n+c-j-1$. 
    Thus, by induction, $X = V_+(g_1, \cdots, g_j, g_{j+1},\cdots, g_c)$ has dimension $\geq n + 1$. This leads to a contradiction, since $X$ is a set-theoretic complete intersection of codimension $c$.
  \end{proof}

  \begin{proposition}\label{prop41}
    The subvariety $Y\subset \mathbb{P}V_{\textbf{F}}$ is a set-theoretic complete intersection of codimension $c$ if and only if $b\geq c-1$. Furthermore, if $b = c$, then $M(\textbf{F})$ is quadratic.
  \end{proposition}
  \begin{proof}
    Let $H$ be the hyperplane of $\mathbb{P}V_{\textbf{F}}$ defined by $w_0 = 0$, and let $U$ be its complement.
    Note that $H\cap Y$ is a cone over $\textbf{Bs}(\textbf{F})\subset \mathbb{P}W \subset \mathbb{P}V_{\textbf{F}}$, so $\operatorname{dim}(H\cap Y) = n + c -b -1$ and $U \cap Y \simeq \mathbb{C}^n$.
    By Lemma \ref{lemma41}, the finite sequence $(f_1, \cdots, f_c)$ is regular if and only if the subvariety $Y\subset \mathbb{P}V_{\textbf{F}}$ is a set-theoretic complete intersection of codimension $c$, which is equivalent to $\operatorname{dim}(Y) = n$.
    Thus, we have $\operatorname{dim}(H\cap Y) = n + c -b -1 \leq n$, which is equivalent to $b \geq c-1$. If $b = c$, then $\operatorname{dim}(H\cap Y) = n-1$, $H \cap Y$ is a divisor of $Y$. As $U \cap Y \simeq \mathbb{C}^n$ is irreducible, $Y$ is irreducible.
    Therefore, $Y = M(\textbf{F})$.
  \end{proof}

  \begin{cor}
    If $\textbf{Bs}(\textbf{F})$ is a set-theoretic complete intersection of codimension $c$, then $M(\textbf{F})$ is a set-theoretic complete intersection of codimension $c$.
  \end{cor}
  \begin{proof}
    By Proposition \ref{prop41}, $Y = M(\textbf{F})$, and $Y$ is a set-theoretic complete intersection of dimension $n$.
  \end{proof}

  \begin{lemma}\label{lemstab}
    The subvariety $Y$ is stable under the action of $W$ on $\mathbb{P}V_{\textbf{F}}$ defined in Theorem \ref{thmfh20}. To be more precise, 
    $$
    g_v\cdot w = [w_0:b_1:\cdots:b_n:b_{n+1}:\cdots:b_{N}],
    $$ for any $v = (v_1, \cdots, v_n) \in W$ and $w = [w_0:w_1:\cdots:w_n:w_{n+1}:\cdots:w_{N}]\in \mathbb{P}V_{\textbf{F}}$,
    where
    \begin{align*}
      b_i & = w_i + w_0 v_i, \; 1\leq i \leq n; \\
      b_{n+j} &= w_{n+j} + 2v\cdot A_j \cdot (w_1, \cdots, w_n)^{t} + w_0Q_j(v), \; 1\leq j \leq c,
    \end{align*}
    and $A_j$ is the symmetric matrix corresponding to the quadric polynomial $Q_j$.
  \end{lemma}
  \begin{proof}
    For any $w = [w_0:w_1:\cdots:w_n:w_{n+1}:\cdots:w_{N}]\in Y$, $\overline{w} := (w_1, \cdots, w_n)$, we have
    \begin{align*}
      & w_0b_{n+j} - Q_j(b_1, \cdots, b_n) \\
      = & w_0 w_{n+j} + 2w_0(v_1, \cdots, v_n)\cdot A_j \cdot \overline{w}^{t}+ Q_j(w_0v_1, \cdots, w_0v_n)- Q_j(b_1,\cdots, b_n)\\
      = & Q_j(w_1, \cdots, w_n) + (w_0v_1, \cdots, w_0v_n)\cdot A_j \cdot \overline{w}^{t}+ v \cdot A_j \cdot (w_0v_1, \cdots, w_0v_n)^{t} \\
      & +Q_j(w_0v_1, \cdots, w_0v_n)- Q_j(b_1,\cdots, b_n)\\
      = & (w_1+w_0v_1, \cdots, w_n + w_0v_n) \cdot A_j \cdot (w_1+w_0v_1, \cdots, w_n + w_0v_n)^{t} - Q_j(b_1,\cdots, b_n) \\
      = & 0,
    \end{align*}
    for $1\leq j \leq c$.
    This completes the proof.
  \end{proof}

  \begin{proposition}\label{bp}
    For a general smooth point $x\in M(\textbf{F})$, we have
    \begin{enumerate}
      \item $Bs(\textbf{F}) = \mathcal{L}_x(Y)$;
      \item $\mathcal{L}_x(M(\textbf{F})) = \mathcal{L}_x(Y)$;
    \end{enumerate}
  \end{proposition}
  \begin{proof}
    We divide the proof into three steps:

    Step 1, for any $y \in O_x$, we have $\mathcal{L}_y(Y) = \mathcal{L}_x(Y)$, where $O_x$ is the orbit of $x$ in $M(\textbf{F})$ under the action of $W$ defined in Lemma
    \ref{lemstab}:

    We only need to show that $g_u\cdot l$ is also a line in $Y$ for all $u \in W$ and line $l$ through $v, w \in Y$,
    \[
        l =\overline{vw}= \{[pv_0+qw_0: \cdots: pv_N + qw_N]\mid [p:q]\in \mathbb{P}^1\}\subset Y.
    \]
    For $u = (u_1, \cdots, u_n) \in W$, we have
    \[
      g_u\cdot l = \{[pv_0+qw_0: z_1: \cdots: z_N]\mid [p:q]\in \mathbb{P}^1\},
    \]
    where
    \begin{align*}
      z_i = & p(v_i+v_0u_i) + q(w_i+w_0u_i) = p(g_u\cdot v)_i + q(g_u\cdot w)_i, \; \forall 1\leq i\leq n;\\
      z_{n+j} = & (pv_{n+j}+qw_{n+j}) + 2 (u_1, \cdots, u_n)A_j(pv_1+qw_1, \cdots, pv_n+qw_n)^t \\
      &+ (pv_0+qw_0)Q_j(u) \\
      = &p(v_{n+j} + 2 (u_1, \cdots, u_n)A_j(v_0, \cdots, v_n)^t + v_0Q_j(u)) +\\
      & q(w_{n+j} + 2 (u_1, \cdots, u_n)A_j(w_0, \cdots, w_n)^t+ w_0Q_j(u))\\
      = &p(g_u\cdot v)_{n+j} + q(g_u\cdot w)_{n+j}.
    \end{align*}
    Therefore, $g_u\cdot l$ is a line in $Y$ through $g_u\cdot v$ and $g_u\cdot w$.

    Step 2, we can assume that $x=[1: 0: \cdots:0]\in M(\textbf{F})$, $v= [v_0: v_1: \cdots: v_N] \in Y$. If the line $l = \overline{xv} \in \mathcal{L}_x(Y)$, we have
    \[
      (p+qv_0)qv_{n+j} = q^2Q_j(v_1, \cdots, v_n), \; [p:q]\in \mathbb{P}^1\; \forall 1\leq j \leq c;
    \]
    therefore, $v_{n+j} = 0$ and $[v_1: \cdots: v_n]\in \textbf{Bs}(\textbf{F})$. Then the line $l$ is of the form
    \[
      l = \{[p: qv_1: \cdots: qv_n: 0:\cdots:0]\mid [p:q]\in \mathbb{P}^1 \}.
    \]
    Therefore, the rational map $\tau_x$ defined in \cite{H01}
    \[
      \tau_x : \mathcal{L}_x(Y) \to \textbf{Bs}(\textbf{F})\subset \mathbb{P}T_xY, \; [l]\mapsto \text{tangent direction at } x,
    \]
    is a regular morphism.
    We have the inverse morphism
    \[
      \phi: \textbf{Bs}(\textbf{F}) \to \mathcal{L}_x(Y), v\mapsto [l_v],
    \]
    where $l_v = \{[p:qv_1:\cdots:qv_n:0:\cdots :0]\mid [p:q]\in \mathbb{P}^1 \} \subset Y$.
    Hence, we have $Bs(\textbf{F}) = \mathcal{L}_x(Y)$.

    Step 3, if $Y = M(\textbf{F})$, then $\mathcal{L}_x(M(\textbf{F})) = \mathcal{L}_x(Y)$. Assume now $M(\textbf{F}) \subsetneq Y$, write $Y = M(\textbf{F}) \cup Y^{\prime}$, $M(\textbf{F})\nsubseteq Y^{\prime}$. If there exists $[l]\in \mathcal{L}_x(Y)$ such that $[l]\notin \mathcal{L}_x(M(\textbf{F}))$, $l\cap M(\textbf{F})$ is a finite set and $x\in l \subset Y^{\prime}$. For any $u\in W$, from Lemma \ref{lemstab} and Theorem \ref{thmfh20} we have
    \[
      g_u\cdot x \in g_u\cdot l \subset Y^{\prime}.
    \]
    According to Theorem \ref{thmfh20}, the orbit of $x$ is an open dense subset of $M(\textbf{F})$, hence $M(\textbf{F})\subset Y^{\prime}$, this leads to a contradiction. Therefore, we have $\mathcal{L}_x(M(\textbf{F})) = \mathcal{L}_x(Y)$.
  \end{proof}

  \begin{thm}\label{thm41}
    Let $M(\textbf{F})\subset \mathbb{P}V_{\textbf{F}}$ be an Euler-symmetric projective variety of dimension $n$ corresponding to the symbol system $\textbf{F}$ of rank $2$ with $F^2 = \left\langle Q_1, \cdots, Q_c\right\rangle$ of dimension $c$. Let $\textbf{Bs}(\textbf{F})\subset \mathbb{P}W$ be the base loci of the symbol system $\textbf{F}$, and $b = \operatorname{codim}_{\mathbb{P}W}(\textbf{Bs}(\textbf{F}))$. For a general smooth point $x\in M(\textbf{F})$, the following are equivalent:
    \begin{enumerate}
      \item $M(\textbf{F})$ is quadratic;
      \item $\textbf{Bs}(\textbf{F})\subset \mathbb{P}W$ is a set-theoretic complete intersection of codimension $c$;
      \item $\mathcal{L}_x(M(\textbf{F}))$ is a set-theoretic complete intersection of codimension $c$;
      \item $b = c$;
    \end{enumerate}
  \end{thm}
  \begin{proof}
    Proposition \ref{bp} implies that $(2)$ and $(3)$ are equivalent. $(2)$ implies $(4)$ is obvious. $(4)$ implies $(1)$ by Proposition \ref{prop41}.

    If $M(\textbf{F})$ is quadratic, then $\operatorname{dim}(Y) = n$ and irreducible. Let $H$ be the hyperplane of $\mathbb{P}V_{\textbf{F}}$ defined by $w_0=0$, and let $U$ be its complement. Since $U \cap Y \simeq \mathbb{C}^n$, we have $\operatorname{dim}(H\cap Y) = n-1 = n + c -b -1$. Therefore, $\operatorname{dim}(\textbf{Bs}(\textbf{F})) = n -1 -c$. Since $\textbf{Bs}(\textbf{F}) = V_+(Q_1, \cdots, Q_c) \subset \mathbb{P}^{n-1}$, $\textbf{Bs}(\textbf{F})$ is a set-theoretic complete intersection of codimension $c$.
  \end{proof}

  \begin{exam}
    If $F^2 = \left\langle x_1^2, x_2^2, x_1x_2 \right\rangle \subset \mathbb{C}[x_1, x_2, x_3, x_4]$, then $M(\textbf{F}) \subsetneq Y$ and $Y$ is a set-theoretic complete intersection of codimension $c = 3$, $\textbf{Bs}(\textbf{F}) = \mathbb{P}^1\subset \mathbb{P}W = \mathbb{P}^3$, $b = c-1$.
  \end{exam}

  \begin{thm}\label{thm42}
     Let $M(\textbf{F})\subset \mathbb{P}V_{\textbf{F}}$ be an Euler-symmetric projective variety of dimension $n$ corresponding to the symbol system $\textbf{F}$ of rank $2$ with $F^2 = \left\langle Q_1, \cdots, Q_c\right\rangle$ of dimension $c$. Let $\textbf{Bs}(\textbf{F})\subset \mathbb{P}W$ be the base loci of the symbol system $\textbf{F}$. The following statements are equivalent: 
     %For a general smooth point $x\in M(\textbf{F})$,
    \begin{enumerate}
      \item $M(\textbf{F})\subset \mathbb{P}V_{\textbf{F}}$ is a complete intersection of codimension $c$;
      \item $\textbf{Bs}(F) \subset \mathbb{P}W$ is a set-theoretic complete intersection of codimension $c$;
      \item The finite sequence $(Q_1, \cdots, Q_c)$ is a regular sequence in $\mathbb{C}[x_1, \cdots, x_n]$.
    \end{enumerate}
  \end{thm}
  \begin{proof}
    By Lemma \ref{lemma41}, $(2)$ is equivalent to $(3)$. Now we prove $(1)$ is equivalent to $(2)$.
    \begin{enumerate}[(i)]
      \item Suppose $M(\textbf{F})$ is a complete intersection. Let $\mathcal{I} = I(M(\textbf{F})) = (g_1, \cdots, g_c)$. Since $f_i \in \mathcal{I}$ and $M(\textbf{F})$ nondegenerate, we have $\operatorname{deg}(g_i)\geq 2$ and $\left\langle f_1, \cdots, f_c \right\rangle \subset \left\langle g_1, \cdots, g_c \right\rangle$. Therefore, $M(\textbf{F}) = Y$. Then $\textbf{Bs}(\textbf{F})$ is a set-theoretic complete intersection of codimension $c$ in $\mathbb{P}W$ by Theorem \ref{thm41}.

      \item Suppose $\textbf{Bs}(\textbf{F})$ is a set-theoretic complete intersection of codimension $c$, then $Y = M(\textbf{F})$ by Theorem \ref{thm41}. If $M(\textbf{F})$ is not a complete intersection, then there is a nonzero homogeneous polynomial $g(w_0, \cdots, w_N)\in \mathcal{I}\backslash (f_1, \cdots, f_c)$ of degree $k$. Since $(f_1, \cdots, f_c) \subset \mathcal{I}$, we can assume that
      \begin{align*}
        g = & \displaystyle\sum_{i =1}^{j}w_0^ig_i(w_1, \cdots, w_n) + g_0(w_1,, \cdots, w_N).
      \end{align*}
      By the definition of $M(\textbf{F})$, $g \in \mathcal{I}$ is equivalent to $g_0 \in \mathcal{I}$ and $g_i \equiv 0$, that is to say, $g = g_0$ is of degree $k$. We have a homogeneous decomposition of $g$ as following
      \[
        g = \sum h_{i,j}(w_1, \cdots, w_n, w_{n+1}, \cdots, w_N),
      \]
      where $h_{i,j}(w_1, \cdots, w_n, w_{n+1}, \cdots, w_N)$ are homogeneous polynomials of degree $i$ w.r.t. $w_1, \cdots, w_n$,
      and are homogeneous polynomials of degree $j$ w.r.t. $w_{n+1}, \cdots, w_{N}$, respectively. Again $g \in \mathcal{I}$ is equivalent to $h_{i,j} \in \mathcal{I}$. If $h_{i,0} \in \mathcal{I}$, then $h_{i,0}\equiv 0$.

      Therefore, we can also assume that $$g = h_{i,j}(w_1, \cdots, w_n, w_{n+1}, \cdots, w_N) \in \mathcal{I}\backslash (f_1, \cdots, f_c) $$ for some $j > 0$, and $w_{n+i} \nmid g, \; \forall\; 1 \leq i \leq c$.
      Thus, we get $$g(x_1, \cdots, x_n, Q_1, \cdots, Q_c) \equiv 0.$$
      By reassigning the index of $\{w_{n+1}, \cdots, w_{n+c}\}$ if necessary, we can rewrite $g$ into the form,
      \[
        g = \sum_{t =1}^l w_{n+t}F_t(w_1, \cdots, w_n, w_{n+t}, \cdots, w_{n+c}),
      \]
      with $F_t \notin (w_{n+1}, \cdots, w_{n+t-1})$ and $l \geq 2$. 
      Therefore, 
      \[
        \sum_{t =1}^l Q_{t}F_t(x_1, \cdots, x_n, Q_{t}, \cdots, Q_{c}) = 0.
      \]
      This implies that the image of $Q_l$ in $\mathbb{C}[x_1, \cdots, x_n]/(Q_1, \cdots, Q_{l-1})$ is a zero-divisor, which contradicts the fact that the sequence $(Q_1, \cdots, Q_{c})$ is a regular sequence.
    \end{enumerate}
	This concludes the proof.
  \end{proof}

  \begin{rmk}
      (1) Theorem \ref{thm2} follows from Theorem \ref{thm41} and Theorem \ref{thm42}. \\
      (2) For a smooth variety which is covered by lines, Theorem 2.4 of \cite{Rasso13} proved a similar result.
  \end{rmk}

  \begin{exam}\label{exam41}
    Let $F^2 = \left\langle x_1^2, \cdots, x_i^2 \right\rangle$, for $1 \leq i \leq n$. $\textbf{F} = F^0 \oplus F^1 \oplus F^2$, then $M(\textbf{F})$ is a complete intersection of codimension $i$.
  \end{exam}

  \begin{exam}
    Let $X$ be an Euler-symmetric projective surface. If $X$ is a complete intersection, then the corresponding symbol system $\textbf{F}$ is of rank $2$, and it is exactly one of the following:
    \begin{enumerate}
      \item $F^2 = \left\langle Q \right\rangle $, where $Q$ is a homogeneous polynomial of degree $2$ in $\mathbb{C}[x_1, x_2]$. In this case, $X$ is smooth if $Q$ is not a square, and $X$ has only one singular point if $Q$ is square. 
      \item $F^2 =\left\langle x_1^2, px_1x_2 + qx_2^2 \right\rangle $, where $q\in \mathbb{C}^{\times}$ and $p\in \mathbb{C}$. In this case, $X$ is non-normal and $\operatorname{Sing}(X) = \mathbb{P}^1$.
    \end{enumerate}
  \end{exam}

  \begin{exam}
    Let $F^2 = \left\langle \sum_{i=1}^n x_i^2, \sum_{i=1}^n \lambda_i x_i^2 \right\rangle$, $\lambda_i \neq \lambda_j$, for $i\neq j$. According to Proposition 2.1 in \cite{reid72}, $\textbf{Bs}(\textbf{F})$ is smooth and of codimension $2$ in $\mathbb{P}^{n-1}$, if $n\geq 3$. By Theorem \ref{thm2}, $M(\textbf{F})$ is a complete intersection, hence not smooth. Proposition 4.4 in \cite{FH20} says that if $M(\textbf{F})$ is nonsingular, then the base loci $\textbf{Bs}(\textbf{F})$ is nonsingular. This example implies that the converse is false.
  \end{exam}

  \begin{exam}
    Let $F^2 = \left\langle Q_1, \cdots, Q_c \right\rangle$, and $\textbf{Bs}(\textbf{F})$ be a set-theoretic complete intersection. We can easily see that $$\operatorname{Sing}(M(\textbf{F})) \supseteq \mathbb{P}^{c-1} = \{[0:\cdots :0:w_{n+1}:\cdots:w_{n+c}]\in \mathbb{P}^{n+c}\},$$ if $c\geq 2$.
    This is another way to see that an Euler-symmetric projective variety which is a complete intersection is not smooth.
  \end{exam}

  \begin{exam}
    Let $F^2 = \left\langle x_1x_5-x_2^2, x_1x_6-x_3^2, x_1x_7-x_2x_3 \right\rangle$, then $\textbf{Bs}(\textbf{F})\subset \mathbb{P}^{6}$ is a set-theoretic complete intersection of codimension $3$, and the corresponding Euler-symmetric variety $M(\textbf{F})\subset \mathbb{P}^{10}$ is a complete intersection of codimension $3$. But $I(\textbf{Bs}(\textbf{F})) = (x_1x_5-x_2^2, x_1x_6-x_3^2, x_1x_7-x_2x_3, x_3x_5-x_2x_7, x_2x_6-x_3x_7)$ and $x_3x_5-x_2x_7 \notin (x_1x_5-x_2^2, x_1x_6-x_3^2, x_1x_7-x_2x_3)$. This implies that $\textbf{Bs}(\textbf{F})\subset \mathbb{P}^{6}$ is not a complete intersection.
  \end{exam}

  \begin{exam}
    Let $F^2 = \left\langle x_1x_3-x_2^2, x_3x_5-x_4^2, x_1x_5-2x_2x_4 +x_3^2 \right\rangle$, then $\textbf{Bs}(\textbf{F})$ is the rational normal quartic curve in $\mathbb{P}^{4}$ by Theorem 1 of \cite{torrente2015rational}. By a direct calculation, $M(F)\subset \mathbb{P}V_{\textbf{F}} = \mathbb{P}^7$ is a complete intersection, i.e.,
    \[
      \mathcal{I} = (w_0w_6-w_1w_3+w_2^2, w_0w_7-w_3w_5+ w_4^2, w_0w_7 -w_1w_5+2w_2w_4-w_3^2).
    \]
  \end{exam}

  \begin{exam}
    Let $F^2 = \left\langle x_1^2, x_1x_2 \right\rangle$, then $M(\textbf{F}) \subset \mathbb{P}^4$ is rational normal scroll of degree $3$. It isn't a complete intersection by Theorem \ref{thm42},
    but it is a set-theoretic complete intersection. For more details, we refer readers to \cite{buadescu2010grothendieck} and \cite{robbiano1983set}.
  \end{exam}

\section*{Acknowledgement}
I'm greatly indebted to my advisor Professor Baohua Fu for illuminating discussions, guidance and revising this paper.  I also want to thank Zheng Xu and Renjie Lyu for helpful communications. Last but not least, I want to thank the referees for the valuable comments and suggestions, which do improve the presentation of this paper much.

\bibliography{Euler-symmetric}
\bibliographystyle{abbrv}

\end{document}